\documentclass[11pt,reqno,a4paper]{amsart}

\usepackage{a4wide}
\usepackage{amssymb}
\usepackage{amsfonts}
\usepackage{enumerate}
\usepackage{amsmath}

\newtheorem{theorem}{Theorem}

\newtheorem{corollary}[theorem]{Corollary}

\newtheorem{definition}[theorem]{Definition}

\newtheorem{lemma}[theorem]{Lemma}

\newtheorem{remark}[theorem]{Remark}

\def\eps{\varepsilon}

\def\N{\mathbb{N}}

\def\R{\mathbb{R}}
\def\P{\mathbb{P}}
\def\E{\mathbb{E}}
\def\T{\mathbb{T}}
\def\om{\omega}

\newcommand{\pf}{f_*\mu}

\begin{document}

\title{Hitting time statistics for observations of dynamical systems}
\author{J\'er\^ome Rousseau}

\begin{abstract}
In this paper we study the distribution of hitting and return times for observations of dynamical systems.
We apply this results to get an exponential law for the distribution of hitting and return times for rapidly mixing random dynamical systems. In particular, it allows us to obtain an exponential law for random expanding maps, random circle maps expanding in average and randomly perturbed dynamical systems.
\end{abstract}


\address{J\'er\^ome Rousseau, Departamento de Matem\'atica, Universidade Federal da Bahia\\
Av. Ademar de Barros s/n, 40170-110 Salvador, Brazil}
\email{jerome.rousseau@ufba.br} 
\urladdr{http://www.sd.mat.ufba.br/~jerome.rousseau}
\thanks{This work was partially supported by FAPESB and CNPq}
\maketitle

%%%%%%%%%
%%%%%%%%%%%%%%%%%
\section{Introduction}
%%%%%%%%%%%%%%%%%%%%%%%
%%%%%%%%%%%%%%%%%%%%%%

When trying to study the reality, experimentalists must often approximate their system or use a simplified model to get a lower dimensional system more suitable to analyze.
In the same philosophy, when working on a high dimensional system, experimentalists generally are just interested in gauging different quantities (temperature, pressure, wind speed, wave height,...) or just know a measurement or observation of the system and tried to use theses observations to get informations on the whole system. 
Following these ideas, a Platonic formalism of dynamical system is given in \cite{MR2031277}. 
More precisely, they want to determine what information on an attractor of a high dimensional dynamical system one can learn by knowing only its image under a function taking values in a lower dimensional space, called an observation. 
To investigate the statistical properties of dynamical systems, the quantitative description of Poincar\'e recurrence behaviour plays an essential part and has been widely studied (see e.g. the review of Saussol \cite{MR2568049}).
Our aim in this article is to examine Poincar\'e recurrence for observations of dynamical systems.

The study of quantitative Poincar\'e recurrence for observations began with the work of Boshernitzan \cite{MR1231839} where it was shown that for a dynamical system $(X,T,\mu)$ and an observation $f$ from $X$ to a metric space $(Y,d)$, if the $\alpha$-dimensional Hausdorff measure is $\sigma$-finite on $Y$ then
\begin{equation}\liminf_{n\rightarrow\infty}n^{1/\alpha}d\left(f(x),f(T^nx)\right)<\infty\qquad\textrm{for $\mu$-almost every $x$}.\end{equation}
Following this work, the definition of return time for the observation was given in \cite{prfo} and the asymptotic behaviour of the return time was analyzed. More precisely, for a measurable function $f:X\rightarrow Y$, a point $x\in X$ and $r>0$, the return time for the observation is defined by:
\[\tau_r^f(x):=\inf\left\{k\in\N^*:\,f(T^kx)\in B(f(x),r)\right\},\]
and it was proved, under some aperiodicity condition, that if the system is rapidly mixing and if the observation $f$ is Lipschitz then $\tau_r^f(x)\underset{r\rightarrow0}{\sim} r^{-d}$ where $d$ is the pointwise dimension of the pushforward measure $f_*\mu$. In \cite{zbMATH06144604}, these results were extend to continuous time and applied to the geodesic flow.

In \cite{prfo}, a short example was given to remark that one can obtain results on the quantitative study of Poincar\'e recurrence for random dynamical systems using the study of recurrence for observations of dynamical systems. This idea was developed in \cite{MR2773129} where, to the best of our knowledge, for the first time annealed and quenched return time for random dynamical systems were defined and studied. It was proved that for super-polynomially mixing random dynamical systems presenting some kind of aperiodicity the random recurrence rates are equal to the pointwise dimensions of the stationary measure (recently the same type of results were obtained in \cite{ArbietoJS} for random hitting time). 

The distribution of the return time and hitting time statistics are another aspect of Poincar\'e recurrence which has been extensively considered for deterministic dynamical systems (one can see the reviews \cite{MR1776758,MR1835750,MR2568049,review-haydn}) and an exponential law was proved for numerous systems with chaotic behaviour (see e.g. \cite{MR1185337,MR1600931,MR1786719,MR2746177,MR2284890,MR1736991,MR1245828,MR1125886,zbMATH05548750}). Recently, a relation between hitting times statistics and extreme values theory has been made by Freitas, Freitas and Todd \cite{MR2749711,MR2639719}. 
One can also mention the work of Keller where spectral perturbation is used to obtain exponential hitting time distributions \cite{MR2903242,MR2535206}.

In the past few months, a couple of papers appeared on the laws of rare events for random dynamical systems. Indeed, in \cite{aytacFV} they link the distribution of hitting time and extreme value laws for randomly perturbed dynamical systems, study the convergence of rare events point process, and prove an exponential law for some randomly perturbed dynamical systems. An exponential distribution for hitting time is also proved in \cite{RSV} for rapidly mixing random subshift of finite type and random expanding maps. We emphasize that the exponential law is not proved with respect to the same measure in both case, the product measure of the skew-product is considered in \cite{aytacFV} while \cite{RSV} works with the sample measures.

Our aim in this paper is to continue the quantitive study of Poincar\'e recurrence for observations, and thus the quantitative study of recurrence of random dynamical systems, investigating the distribution of hitting and return time for the observation. An exponential law for the distribution of the hitting and return time for the observation is given in Section~\ref{sectionobservation} and proved in Section~\ref{sectionproofobservation}, for rapidly mixing dynamical systems, presenting some kind of aperiodicity. This allows us to obtain in Section~\ref{applications}, while the proofs are given in Section~\ref{proofrds}, an exponential law for super-polynomially mixing random dynamical systems and apply this result to random expanding maps, random circle maps expanding in average and randomly perturbed dynamical systems.

%%%%%%%%%%%%%%%%%%%%%%%%%%%
%%%%%%%%%%%%%%%%%%%%
\section{Exponential law for observations of dynamical systems}\label{sectionobservation}
%%%%%%%%%%%%%%%%%%%%%%
%%%%%%%%%%%%%%%%%%%%%%%

Let $(X,\mathcal{A},\mu,T)$ be a measure preserving system (m.p.s.) i.e. $\mathcal{A}$ is a $\sigma$-algebra, $\mu$ is a probability measure on $(X,\mathcal{A})$ and $\mu$ is invariant by $T$ (i.e $\mu(T^{-1}A)=\mu(A)$ for all $A\in\mathcal{A}$) where $T:X\rightarrow X$. We assume that $X$ is a metric space and $\mathcal{A}$ is its Borel $\sigma$-algebra. 

We introduce the hitting and return time for the observation and its associated recurrence rates.
\begin{definition}Let $f:X\rightarrow \R^N$ be a measurable function, called observation, and $A\subset \R^N$, we define for $x\in X$  the hitting time for the observation of $x$ in $A$:
\[\tau_A^f(x):=\inf\left\{k\in\N^*:\,f(T^kx)\in A\right\}.\]
When we study hitting time in ball, we define for $x_0\in X$  and $x\in X$ the hitting time for the observation:
\[\tau_r^f(x,x_0):=\inf\left\{k\in\N^*:\,f(T^kx)\in B\left(f(x_0),r\right)\right\}\]
where $B(y,r)$ denoted the ball centered in $y$ with radius $r$. Also, we define for $x\in X$ the return time for the observation:
\[\tau_r^f(x):=\inf\left\{k\in\N^*:\,f(T^kx)\in B\left(f(x),r\right)\right\}.\]
We then define the lower and upper recurrence rate for the observation:
\[\underline{R}^f(x):=\liminf_{r\rightarrow0}\frac{\log\tau_r^f(x)}{-\log r}\qquad\overline{R}^f(x):=\limsup_{r\rightarrow0}\frac{\log\tau_r^f(x)}{-\log r}.\]
\end{definition}
To obtain optimal results on the return and hitting times for the observation we need to assume that the system presents some kind of aperiodicity:
\begin{definition}
A m.p.s. $(X,\mathcal{A},\mu,T)$ is called $\mu$-almost aperiodic for the observation $f$ if
\[\mu\left(x\in X:  \exists n\in\N^*\textrm{ such that }f(T^nx)=f(x)\right)=0.\]
\end{definition}
We emphasize that this condition can be remove introducing non-instantaneous return times \cite{prfo}. 

Since we want to link the behaviour of the return time to the behaviour of the measure of shrinking sets, we remind the definition of \emph{the lower and upper pointwise or local dimension} of a Borel probability measure $\nu$ on Y at a point $y\in Y$
\[\underline{d}_\nu(y)=\underset{r\rightarrow0}{\underline\lim}\frac{\log\nu\left(B\left(y,r\right)\right)}{\log r}\qquad\textrm{and}\qquad\overline{d}_\nu(y)=\underset{r\rightarrow0}{\overline\lim}\frac{\log\nu\left(B\left(y,r\right)\right)}{\log r}.\]
 We also remind that the conditional measure is
\[\nu_A(B)=\frac{\nu(A\cap B)}{\nu(A)}\]
and the pushforward measure is $f_*\nu(.):=\nu(f^{-1}(.))$. 

In order to study the behaviour of the return time, we need a rapid mixing condition:
\begin{definition}\label{defdecay}
$(X,T,\mu)$ has a super-polynomial decay of correlations if, for all $\psi$ Lipschitz function from $X$ to $\R$, for all $\phi$ measurable bounded function from $X$ to $\R$ and for all $n\in\N^*$,  we have:
\[\left|\int_X\psi.\phi\circ T^n\, d\mu-\int_X \psi d\mu\int_X\phi d \mu\right|\leq\|\psi\|_{Lip}\|\phi\|_\infty\theta_n\]
with $\lim_{n\rightarrow\infty}\theta_n n^p=0$ for all $p>0$.
\end{definition}

In \cite{prfo}, the authors proved that the recurrence rates for the observation are linked to the local dimension of the pushforward measure:
\begin{theorem}\cite{prfo}\label{theoprfo}
Let $(X,\mathcal{A},\mu,T)$ be a m.p.s and $f:X\rightarrow \R^N$ a measurable observation such that the system is $\mu$-almost aperiodic for $f$. Then, for $\mu$-almost every $x\in X$
\[\underline{R}^f(x)\leq \underline{d}_{f_*\mu}(f(x)) \qquad \textrm{and}\qquad\overline{R}^f(x)\leq \overline{d}_{f_*\mu}(f(x)).\]
Moreover, if the system has a super-polynomial decay of correlations and $f$ is Lipschitz then
\[\underline{R}^f(x)= \underline{d}_{f_*\mu}(f(x)) \qquad \textrm{and}\qquad\overline{R}^f(x)= \overline{d}_{f_*\mu}(f(x))\]
for $\mu$-almost every $x\in X$ such that $\underline{d}_{f_*\mu}(f(x))>0$.
\end{theorem}
One can observe that in \cite{prfo}, they assumed that the observables $\phi$, in Definition~\ref{defdecay}, are Lipschitz functions. A simple modification of their proof allows us to state their theorem with measurable bounded functions, a necessary assumption in the proof of our main theorem.

To obtain an information on the fluctuation of the return time we will need that the system fulfills the assumptions of the previous theorem and we will also need an hypothesis on the measure:
\begin{enumerate}[I)]
\item\label{hypcouronne} For $f_*\mu$-almost every $y\in \R^n$, there exist $a>0$ and $b\geq 0$ such that
\begin{equation}\label{eqcour}
f_*\mu\left(B\left(y,r\right)\backslash B\left(y,r-\rho\right)\right)\leq r^{-b}\rho^a
\end{equation}
for any $r>0$ sufficiently small and any $0<\rho<r$.
\end{enumerate}
In Section~\ref{applications} we will apply our results to random dynamical systems and give some examples where all the assumptions are fulfilled. Some examples of measure which fulfilled the hypothesis~\eqref{hypcouronne} are given in Lemma 44 of \cite{MR2568049}.

Our theorem on the fluctuations of the return time for the observations is the following:
\begin{theorem}\label{thprinc}
Let $(X,\mathcal{A},\mu,T)$ be a m.p.s with a super-polynomial decay of correlations and $f:X\rightarrow \R^N$ a Lipschitz observation such that the system is $\mu$-almost aperiodic for the observation $f$. If hypothesis \eqref{hypcouronne} is satisfied, then for every $t\geq0$ and $\mu$-almost every $x_0\in X$ such that $\underline{d}_{f_*\mu}(f(x_0))>0$:
\[\lim_{r\rightarrow0}\mu\left(x\in X, \tau_r^f(x,x_0)>\frac{t}{f_*\mu\left(B\left(f(x_0),r\right)\right)}\right)=e^{-t}\]
and
\[\lim_{r\rightarrow0}\mu_{f^{-1}B\left(f(x_0),r\right)}\left(x\in X, \tau_r^f(x,x_0)>\frac{t}{f_*\mu\left(B\left(f(x_0),r\right)\right)}\right)=e^{-t}.\]
\end{theorem}
One can observe that, for rapidly mixing dynamical system, we can applied this theorem to the observation $f=id$ and we obtain some of the results cited in the introduction and, in particular, a generalization of Theorem 40 of \cite{MR2568049}:

 \begin{corollary}
Let $X\subset \R^N$ and let $(X,\mathcal{A},\mu,T)$ be a m.p.s with a super-polynomial decay of correlations. 
Let us suppose that for $\mu$-almost every $y\in X$, there exist $a>0$ and $b\geq 0$ such that $\mu\left(B\left(y,r\right)\backslash B\left(y,r-\rho\right)\right)\leq r^{-b}\rho^a$
for any $r>0$ sufficiently small and any $0<\rho<r$. 

Then for every $t\geq0$ and for $\mu$-almost every $x_0\in X$ such that $\underline{d}_{\mu}(x_0)>0$:
\[\lim_{r\rightarrow0}\mu\left(x\in X, \tau_r(x,x_0)>\frac{t}{\mu\left(B\left(x_0,r\right)\right)}\right)=e^{-t}\]
and
\[\lim_{r\rightarrow0}\mu_{B\left(x_0,r\right)}\left(x\in X, \tau_r(x,x_0)>\frac{t}{\mu\left(B\left(x_0,r\right)\right)}\right)=e^{-t}\]
where $\tau_r(x,x_0):=\inf\left\{k\in\N^*:\,T^k x\in B\left(x_0,r\right)\right\}$.
\end{corollary}
One can remark that, in this corollary, we do not need the assumption on the $\mu$-almost aperiodicity since the system is mixing.

%%%%%%%%%%%%%%%%%%%
%%%%%%%%%%%%%%%%
\section{Hitting time statistics for random dynamical systems}\label{applications}
%%%%%%%%%%%%
%%%%%%%%%%%

In \cite{prfo}, it was observed that the study of recurrence for observations of dynamical systems allows to 
study recurrence for random dynamical systems. Following this remark, random return 
times for dynamical systems were defined in \cite{MR2773129} and the random recurrence rates were linked to the local dimension of the 
stationary measure. Recently, a few works on this theme are emerging, indeed, random hitting time indicators are 
studied in \cite{ArbietoJS}, \cite{aytacFV} linked the distribution of hitting time for randomly perturbed dynamical 
systems and extreme value laws, and an exponential distribution for hitting time is proved in \cite{RSV} for random 
subshift of finite type.

In this section, we will use our results for observations of dynamical systems to prove an exponential law 
for random dynamical systems presenting some rapidly mixing conditions and apply this result to random expanding maps, random circle maps expanding in average and randomly perturbed dynamical systems.

%%%%%%%%%%%%%
\subsection{Exponential law for random dynamical systems}
%%%%%%%%%%%%

Let $\Omega$ be a metric space and $\mathcal{B}(\Omega)$ its Borelian $\sigma$-algebra. Let $\vartheta:\Omega \to \Omega$ be a measurable transformation on $\Omega$
preserving some probability measure $\P$.  Given a compact metric space $X$ and a family $\mathcal T=(T_\om)_{\om\in\Omega}$ of transformations $T_\om:X \to X$, we say that it defines a \emph{random dynamical system} over $(\Omega,\mathcal{B}(\Omega),\mathbb{P},\vartheta)$ via $T_\om^n=T_{\vartheta^{n-1}(\om)}\circ \dots \circ T_{\vartheta(\om)}\circ T_\om$
for every $n\ge 1$ and $T_\om^0=Id$.

The dynamics of the random dynamical systems generated by $\mathcal T$ over$(\Omega,\mathcal{B}(\Omega),\mathbb{P},\vartheta)$ is given by the skew-product: 
\begin{eqnarray*}
S:   \Omega\times X &  \to &  \Omega \times X \\
      (\om, x) & \to & (\vartheta(\om),T_\om(x)).
\end{eqnarray*}
A probability measure $\mu$ is \emph{invariant} by the random dynamical system if it is $S$-invariant and
$\pi_* \mu=\P$, where $\pi:\Omega\times X \to \Omega$ is the canonical projection. Henceforth, we denote by $\nu$ the marginal of $\mu$ on $X$, i.e. $\nu=\int \mu_\omega \, d\P$ where $(\mu_\omega)_\omega$ denote the decomposition of $\mu$ on $X$, that is, $d\mu(\omega,x)=d\mu_\omega(x)d\P(\omega)$. 

When the skew product invariant measure is a product measure $\mu=\P\otimes\nu$, we will say that $\nu$ is a stationary measure for the random dynamical system. One can observe that it includes the case where the maps $T_{\omega}$ are chosen independently and with the same distribution. 

Now we shall define return and hitting times for random dynamical systems. For a fixed $\omega\in \Omega$ the quenched random hitting time in a measurable subset $A\subset X$ of the random orbit starting from a point $x\in X$ is:
\[\tau_A^{\omega}(x)=\inf\{n>0\,:\, T^n_{\omega}x\in A\}.\]

From now on we assume that $X\subset \R^N$ for some $N\in\N$. For $\omega\in\Omega$ and $x_0\in X$, we are interested in the behavior as $r\rightarrow0$ of the quenched random hitting time of a point $x\in X$ into the open ball $B(x_0,r)$ defined by
\[\tau_{r}^{\omega}(x,x_0):=\inf\left\{n>0: T^n_{\omega}x\in B(x_0,r)\right\}\]
and the quenched random return time of a point $x\in X$ into the open ball $B(x,r)$ defined by
\[\tau_{r}^{\omega}(x):=\inf\left\{n>0: T^n_{\omega}x\in B(x,r)\right\}.\]
As previously, we will deal with systems presenting some kind of aperiodicity:
\begin{definition}
The random dynamical system $\mathcal{T}$ on $X$ over $(\Omega,\mathcal{B}(\Omega),\mathbb{P},\vartheta)$ with an invariant measure $\mu$ is called random-aperiodic if 
\[\mu\left((\omega,x)\in\Omega\times X:\exists n\in \mathbb{N},\,T^n_{\omega}x=x\right)=0.\]
\end{definition}
More details on such systems can be found in the Section 2.3 of \cite{prfo} and the Section 4 and 5 of \cite{MR2773129}. We also need an assumption on the measure and an assumption on the decay of correlations, which will be weaker than assuming super-polynomial decay of correlations for the skew-product:
\begin{enumerate}[a)]
\item For $\nu$-almost every $x\in X$, there exist $a>0$ and $b\geq 0$ such that
\[\nu\left(B\left(x,r\right)\backslash B\left(x,r-\rho\right)\right)\leq r^{-b}\rho^a\]
for any $r>0$ sufficiently small and any $0<\rho<r$. \label{hypcourrds}
\item For all $n\in\mathbb{N}^{*}$, $\psi$ Lipschitz observables from $X$ to $\mathbb{R}$ and $\varphi$ measurable bounded from $\Omega\times X$ to $\mathbb{R}$
$$
\vert\int_{\Omega\times X}\psi(x)\varphi(S^n(\omega,x))d\mu-\int_X\psi d\nu \int_{\Omega\times X} \varphi d\mu\vert\leq \|\psi\|_{Lip}.\| \varphi\|_\infty .\theta_n
$$
with $\lim_{n\rightarrow\infty}\theta_nn^p=0$ for any $p>0$. \label{decayskewrds}
\end{enumerate}
\begin{theorem}\label{theorandomhit}
Let $\mathcal{T}$ be a random dynamical system on $X$ over $(\Omega,\mathcal{B}(\Omega),\mathbb{P},\vartheta)$ with an invariant measure $\mu$. If the random dynamical system is random-aperiodic and satisfied hypothesis \eqref{hypcourrds} and \eqref{decayskewrds} then for every $t\geq0$ and for $\nu$-almost every $x_0\in X$ such that $\underline{d}_{\nu}(x_0)>0$:
\[\lim_{r\rightarrow0}\mu\left((\omega,x)\in \Omega\times X, \tau_{r}^\omega(x,x_0)>\frac{t}{\nu\left(B\left(x_0,r\right)\right)}\right)=e^{-t}\]
and
\[\lim_{r\rightarrow0}\mu_{\Omega\times B\left(x_0,r\right)}\left((\omega,x)\in \Omega\times X, \tau_{r}^\omega(x,x_0)>\frac{t}{\nu\left(B\left(x_0,r\right)\right)}\right)=e^{-t}.\]
\end{theorem}

The basic idea to prove this theorem, which was already used in \cite{prfo,MR2773129}, is to applied Theorem~\ref{thprinc} 
to the dynamical system 
$(\Omega\times X,\mathcal{B}(\Omega\times X),\mu ,S)$ with the specific observation $f$ defined by
\begin{eqnarray*}
f\,\,:\Omega\times X&\longrightarrow& X\\
(\omega,x)&\longmapsto& x.
\end{eqnarray*}
Indeed, with this observation, the hitting time for the observation and the hitting time for the 
random dynamical system are equal:
\[\tau^f_{B(f(x_0,\omega_0),r)}(\omega,x)=\tau^\omega_{B(x_0,r)}(x).\]
The complete proof of the Theorem will be done in Section~\ref{proofrds}.

\begin{remark}
One can observe that this result is complementary to the ones proved in \cite{RSV}. Indeed, in our result only the point $x_0$ is fixed and we proved an exponential law with respect to the invariant measure $\mu$. 

In \cite{RSV}, both the target $x_0$ and the alea $\omega$ are fixed and the exponential law is proved with respect to the sample measures $\mu_\omega$ and also with respect to the marginal $\nu$.

We emphasize that in \cite{aytacFV} also exponential law with respect to the invariant measure $\mu$ is obtained for randomly perturbed dynamical systems.

\end{remark}

For i.i.d. random dynamical systems, to obtain an exponential law, 
we just need to assume a super-polynomial decay of correlations for the random dynamical system, 
i.e. our observables are from $X$ to $\mathbb{R}$, which is a more natural assumption than hypothesis \eqref{decayskewrds}.

More precisely, let $\{T_{\lambda}\}_{\lambda\in\Lambda}$ be a family of transformations 
defined on a compact Riemannian manifold $X$ and let $\mathcal{P}$ be a probability measure on a metric space
$\Lambda$.  We will consider $\mathcal{T}$ a random dynamical system on $X$ over 
$(\Lambda^\N,\mathcal{P}^\N,\sigma)$ with a stationary measure $\nu$, where $\sigma$ is the shift. That is, for an \textit{i.i.d.} stochastic process $\underline{\lambda}=(\lambda_n)_{n\geq 1}\in\Lambda^{\mathbb{N}}$ with common 
distribution $\mathcal{P}$, a random evolution of an initial state $x\in X$  will be:
$$
T^n_{\underline{\lambda}} x=T_{\lambda_{n}}\circ\ldots T_{\lambda_1}x
$$
for every $n\geq 0$.

\begin{definition}
The i.i.d. random dynamical system has a super-polynomial decay of correlations if, for all $n\in\mathbb{N}^{*}$, $\psi$ Lipschitz observables from $X$ to $\mathbb{R}$ and $\varphi$ measurable bounded from $ X$ to $\mathbb{R}$
$$
\vert\int_{\Lambda^\N\times X}\psi(x)\varphi(T^n_{\underline{\lambda}} x)d\mathcal{P}^\N d\nu-\int_X\psi d\nu \int_{X} \varphi d\nu\vert\leq \|\psi\|_{Lip}.\|\varphi\|_\infty .\theta_n
$$
with $\lim_{n\rightarrow\infty}\theta_nn^p=0$ for any $p>0$.
\end{definition}
\begin{theorem}\label{theorandomhitiid}
Let $\mathcal{T}$ be an i.i.d. random dynamical system on $X$ over $(\Lambda^\N,\mathcal{P}^\N,\sigma)$
with a stationary measure $\nu$. If the random dynamical system is random-aperiodic, satisfied 
hypothesis \eqref{hypcourrds} and has a super-polynomial decay of correlations then for every $t\geq0$ and for $\nu$-almost 
every $x_0\in X$ such that $\underline{d}_{\nu}(x_0)>0$:
\[\lim_{r\rightarrow0}\mathcal{P}^\N\otimes\nu\left((\underline{\lambda},x)\in \Lambda^\N\times X, \tau_{r}^{\underline{\lambda}}(x,x_0)>\frac{t}{\nu\left(B\left(x_0,r\right)\right)}\right)=e^{-t}\]
and
\[\lim_{r\rightarrow0}\mathcal{P}^\N\otimes\nu_{\Lambda^\N\times B\left(x_0,r\right)}\left((\underline{\lambda},x)\in \Lambda^\N\times X, \tau_{r}^{\underline{\lambda}}(x,x_0)>\frac{t}{\nu\left(B\left(x_0,r\right)\right)}\right)=e^{-t}.\]
\end{theorem}
\begin{remark}
We emphasize that this result extends the result of \cite{aytacFV} for randomly perturbed dynamical systems. The principal generalization lies in the decay of correlations.

First of all, they need polynomial decay of correlations against $L^1$ observables when here we just need super-polynomial decay of correlations against $L^\infty$ observables. Besides, for the observables $\psi$, we do not assume that indicator functions of balls are bounded in the Banach space.

Moreover, they study randomly perturbed dynamical systems, more precisely, they perturbed an original map with random additive noise when in our setting we can study more general random dynamical systems, as shown in the following examples.
\end{remark}

In the next subsections, we will give examples of random dynamical systems where we can apply our results. The first example was given in \cite{MR2773129} as an example of non-i.i.d. random dynamical systems where the recurrence rates can be computed.

%%%%%%%%%%%%%%%%%%
\subsection{Non-i.i.d. random expanding maps}
%%%%%%%%%%%%%%%%%

Let  $T_1$ and $T_2$ be the two following maps defined on the one-dimensional torus $X=\T^1$:
\[\begin{array}{rcccrcc}
T_1:X&\longrightarrow& X & \textrm{ and }&T_2:X&\longrightarrow& X\\
x&\longmapsto& 2x& & x&\longmapsto& 3x.
\end{array}\]
The dynamic of the random dynamical system is given by the following skew product
\begin{eqnarray*}
S:\Omega\times X&\longrightarrow& \Omega\times X\\
(\omega,x)&\longmapsto&(\vartheta(\omega),T_\omega x)
\end{eqnarray*}
with $\Omega=[0,1]$, $T_\omega=T_1$ if $\omega\in[0,2/5)$, $T_\omega=T_2$ if $\omega\in[2/5,1]$ and where $\vartheta$ is the following piecewise linear map
\[\vartheta(\omega)=\left\{\begin{array}{lll}
2\omega\qquad&\textrm{if}&\omega\in[0,1/5)\\
3\omega-1/5\qquad&\textrm{if}&\omega\in[1/5,2/5)\\
2\omega-4/5\qquad&\textrm{if}&\omega\in[2/5,3/5)\\
3\omega/2-1/2 \qquad&\textrm{if}&\omega\in[3/5,1].\\
\end{array}\right.\]
One can observe that the random orbit is constructed by choosing the map $T_1$ and $T_2$ following a Markov process with the stochastic matrix
\[A=\begin{pmatrix}1/2&1/2\\1/3&2/3\end{pmatrix}.\]
It was proved in \cite{MR2773129} that the associated skew-product is $Leb\otimes Leb$-invariant, is random aperiodic and has an exponential decay of correlations. 

We can verify that Lebesgue measure satisfied hypothesis \eqref{hypcourrds} and thus Theorem~\ref{theorandomhit} applies, i.e. for every $t\geq0$ and for $Leb$-almost 
every $x_0\in \T^1$:
\[\lim_{r\rightarrow0}Leb\otimes Leb\left( \tau_{r}^\omega(x,x_0)>\frac{t}{r}\right)=e^{-t}\]
and
\[\lim_{r\rightarrow0}Leb\otimes Leb_{[0,1]\times B\left(x_0,r\right)}\left( \tau_{r}^\omega(x,x_0)>\frac{t}{r}\right)=e^{-t}.\]

%%%%%%%%%%%%%
\subsection{Random circle maps expanding in average}
%%%%%%%%%%%%%%%%%%%

  Let $\Lambda$ be a metric space with a probability measure $\mathcal{P}$. For every $\lambda\in\Lambda$, let $T_\lambda:\T^1\rightarrow\T^1$ be an application $C^2$ without critical point.
  We will assume that 
  \[\int_\Omega\frac{1}{\inf \vert T_\lambda'\vert}d\mathcal{P}(\lambda)<1\]
  and that
  \[\int_\Omega\left\Vert\frac{T_\lambda''}{( T_\lambda')^2}\right\Vert_\infty d\mathcal{P}(\lambda)<+\infty.\]
  It has been proved (e.g. \cite{stenlundsulku} and references therein), that for the i.i.d. random dynamical system on $\T^1$ over 
$(\Lambda^\N,\mathcal{P}^\N,\sigma)$ there exists an absolutely continuous stationary measure
  and that the random dynamical system has an exponential decay of correlations for H\"older observables
  (nevertheless one can use for example \cite{GRS} to go from H\"older to Lipschitz observables).
  Since absolutely continuous invariant measures satisfy hypothesis~\eqref{hypcourrds}, we obtain by Theorem~\ref{theorandomhitiid} that if the system is 
  random aperiodic, we have an exponential law for the return time and the hitting time. 
  
  For example, one can apply this results for random $\beta$-transformations. More precisely,  for every $\lambda\in\Lambda$, let $T_\lambda:\T^1\rightarrow\T^1$
  such that $T_\lambda x=\beta_\lambda x$ where $\beta_\lambda>1$. To prove, that this system has an exponential law, we just need
  to prove that it is random aperiodic.
  For any $\underline{\lambda}\in\Lambda^\N$ and any $n\in\N$, $T^n_{\underline{\lambda}}$ is a $\beta$-transformation, thus
   \[\textrm{Card} \{x\in\mathbb{T}^1\,:\,T^n_{\underline{\lambda}}x=x\}<+\infty\]
and so for every $\underline{\lambda}\in\Lambda^\N$
\[Leb(\{x\in\mathbb{T}^1\,:\,\exists n\in\mathbb{N}^*\,,\,T^n_{\underline{\lambda}}x=x\})=0.\]
Then, since the stationary measure is absolutely continuous, we obtain that
\begin{equation*}
\mathcal{P}^\N\otimes \nu(\{(\underline{\lambda},x)\in\Lambda^\N\times\mathbb{T}^1\,:\,\exists n\in\N^*\,,\, T^n_{\underline{\lambda}}x=x\})=0
\end{equation*}
  and therefore, one can apply Theorem~\ref{theorandomhitiid}.
  
  %%%%%%%%%%%%%%%%%%
\subsection{Randomly perturbed dynamical systems}
%%%%%%%%%%%%%%%%%%%
In this section, we just give a short overview of some randomly perturbed dynamical systems (see e.g. \cite{MR1233850,zbMATH01791970}) for which our results apply, obtaining a generalization of \cite{aytacFV}.

Let $X$ be a compact Riemannian manifold and let $(X,T,\mu)$ be a deterministic dynamical system. We build our random dynamical systems by perturbing our transformation $T$ with random additive noise. More precisely,  for $\varepsilon>0$, let $\Lambda_\varepsilon$ be a metric space (typically $\Lambda_\varepsilon=B(0,\eps)$) and $\mathcal{P}_\eps$ a probability measure on $\Lambda_\varepsilon$. The family of transformations $\{T_{\lambda}\}_{\lambda\in\Lambda_\eps}$, where $T_\lambda:X\rightarrow X$, are defined by
\[T_\lambda(x)=T(x)+\lambda.\]
We consider $\mathcal{T}$ the i.i.d random dynamical system on $X$ over 
$(\Lambda_\eps^\N,\mathcal{P}_\eps^\N,\sigma)$. 

For $X=\T^d$, it has been proved (e.g. \cite{MR1233850,vi, aytacFV}) that for some expanding and piecewise expanding maps, if $\eps$ is small enough, the random dynamical system admits a stationary measure $\nu_\eps$ absolutely continuous with respect to the Lebesgue measure (thus satisfied hypothesis \eqref{hypcourrds}) and that the system has a super-polynomial decay of correlations.

Moreover, since these systems are random aperiodic \cite{MR2773129}, one can apply Theorem~\ref{theorandomhitiid} and obtain an exponential law.

This gives, for example, an exponential law for perturbation of expanding and piecewise expanding maps of the circle with a finite number of discontinuities (see \cite{vi} or \cite{aytacFV} for precise definitions) and also an exponential law for perturbations of expanding and piecewise expanding maps of $\T^d$ (e.g. \cite{aytacFV,zbMATH01445781}).

%%%%%%%%%%%%%%%
%%%%%%%%%%%%%%%
\section{Proof of the exponential law for observations of dynamical systems}\label{sectionproofobservation}
%%%%%%%%%%%%%%%%%
%%%%%%%%%%%%%%%%%%%
This section is dedicated to the proof of Theorem~\ref{thprinc}.

The proof of Theorem~\ref{thprinc} follows the ideas of \cite{MR1736991,MR2568049}. For a measurable subset $A\subset\R^N$ such that $\pf(A)>0$, we define
\[\delta(A)=\sup_{k\in\N}\left|\mu\left(\tau_A^f>k\right)-\mu_{f^{-1}A}\left(\tau_A^f>k\right)\right|.\]
\begin{lemma}
Let $A\subset\R^N$ such that $\pf(A)>0$. For any $n\in\N$, we have
\[\left|\mu\left(\tau_A^f>n\right)-\left(1-\pf(A)\right)^n\right|\leq\delta(A).\]
\end{lemma}
\begin{proof}Using the fact that $\tau_A^f(x)=\tau_{f^{-1}A}(x):=\inf\{k>0, T^kx\in f^{-1}A\}$, one can apply Lemma 41 of \cite{MR2568049} to the set $f^{-1}A$.
\end{proof}
To prove the exponential law for the hitting and return time, we will need to estimate the decay of $\delta$:
\begin{lemma}\label{lemlimdelta}
Under the assumption of Theorem~\ref{thprinc}, we have
\[\lim_{r\rightarrow0}\delta\left(B\left(f(x_0),r\right)\right)=0\]
for $\mu$-almost every $x_0$ such that $\underline{d}_{f_*\mu}(f(x_0))>0$.
\end{lemma}
\begin{proof}
Let $x_0\in X$, $r>0$, $0<\rho<r$ and $n\in\N$. Let $A=B(f(x_0),r)$, $B_n=\{\tau_A^f>n\}$ and $g\geq n$. Let $\phi:X\rightarrow \R$ a function to be defined later. We have
\begin{eqnarray*}
\left|\mu\left(f^{-1}A\cap B_n\right)-\mu(f^{-1}A)\mu(B_n)\right|&\leq & \left|\mu\left(f^{-1}A\cap B_n\right)-\mu\left(f^{-1}A\cap T^{-g}B_{n-g}\right)\right|\\
& +&\left|\mu\left(f^{-1}A\cap T^{-g}B_{n-g}\right)-\int\phi.1_{B_{n-g}}\circ T^g d\mu\right|\\
&+&\left|\int\phi.1_{B_{n-g}}\circ T^g d\mu-\mu(B_{n-g})\int\phi d\mu\right|\\
&+&\left|\mu(B_{n-g})\int\phi d\mu-\mu(B_{n-g})\mu(f^{-1}A)\right|\\
&+&\left|\mu(B_{n-g})\mu(f^{-1}A)-\mu(f^{-1}A)\mu(B_n)\right|.
\end{eqnarray*}
Let us estimate each term of this inequality. First of all, we use the definition of $B_n$ to estimate the first term:
\begin{eqnarray*}
\left|\mu\left(f^{-1}A\cap B_n\right)-\mu\left(f^{-1}A\cap T^{-g}B_{n-g}\right)\right|&=&\mu\left(f^{-1}A\cap T^{-g}B_{n-g}\cap\{\tau_A^f\leq g\}\right)\\
&\leq& \mu\left(f^{-1}A\cap\{\tau_A^f\leq g\}\right).
\end{eqnarray*}
Let $C=B(f(x_0),r-\rho)$ and define $\phi(x):=\max\left(0, 1-\frac{1}{\rho}d(f(x),C)\right)$. One can observe that $1_{f^{-1}C}\leq \phi\leq 1_{f^{-1}A}$ and that $\phi$ is $\frac{K}{\rho}$-Lipschitz where $K=\max(1, |f|_{Lip})$, thus
\begin{eqnarray*}
\left|\mu\left(f^{-1}A\cap T^{-g}B_{n-g}\right)-\int\phi.1_{B_{n-g}}\circ T^g d\mu\right|&=&\left|\int1_{f^{-1}A}.1_{B_{n-g}}\circ T^g d\mu-\int\phi.1_{B_{n-g}}\circ T^g d\mu\right|\\
&\leq&\int\left|1_{f^{-1}A}-\phi \right|d\mu\\
&\leq&\int1_{f^{-1}A}-1_{f^{-1}C} d\mu=\mu\left(f^{-1}A\backslash f^{-1}C\right)\\
&\leq&\pf\left(A\backslash C\right).
\end{eqnarray*}
The information on the decay of correlation gives us the estimate of the third term:
\begin{eqnarray}
\left|\int\phi.1_{B_{n-g}}\circ T^g d\mu-\mu(B_{n-g})\int\phi d\mu\right|&\leq&\|\phi\|_{Lip}\|1_{B_{n-g}}\|_{\infty}.\theta_g \nonumber\\
&\leq& \frac{K}{\rho}\theta_g.\label{eqdecay}
\end{eqnarray}
The estimate of the fourth term is given using the same idea of the second term:
\begin{eqnarray*}
\left|\mu(B_{n-g})\int\phi d\mu-\mu(B_{n-g})\mu\left(f^{-1}A\right)\right|&\leq& \left|\int\phi-1_{f^{-1}A}d\mu \right|\\
&\leq&\pf\left(A\backslash C\right).
\end{eqnarray*}
Finally, the last term can be estimate using the invariance of the measure and the definition of $B_n$:
\begin{eqnarray*}
\left|\mu(B_{n-g})\mu(f^{-1}A)-\mu(f^{-1}A)\mu(B_n)\right|&=&\mu(f^{-1}A)\left(\mu\left(T^{-g}B_{n-g}\right)-\mu(B_n)\right)\\
&\leq&\mu(f^{-1}A)\mu\left(\tau_A^f\leq g\right).
\end{eqnarray*}
These estimates give us
\begin{equation}\label{eqborne}
\left|\mu_{f^{-1}A}(B_n)-\mu(B_n)\right|\leq\mu_{f^{-1}A}\left(\tau_A^f\leq g\right)+2\frac{\pf(A\backslash C)}{\pf(A)}+\frac{K}{\rho}\frac{\theta_g}{\pf(A)}+\mu\left(\tau_A^f\leq g\right).
\end{equation}
One can observe that this inequality is still satisfied if $n\leq g$. Indeed, when $n\leq g$, we have 
\begin{eqnarray*}
\left|\mu_{f^{-1}A}(B_n)-\mu(B_n)\right|&\leq&\left|1-\mu_{f^{-1}A}(B_n)\right|+\left|1-\mu(B_n)\right|\\
&\leq&\mu_{f^{-1}A}\left(\tau_A^f<n\right)+\mu\left(\tau_A^f<n\right)\\
&\leq&\mu_{f^{-1}A}\left(\tau_A^f\leq g\right)+\mu\left(\tau_A^f\leq g\right).
\end{eqnarray*}
Then, \eqref{eqborne} holds for every $n\in\N$ and gives us an upper bound for $\delta(B(f(x_0),r))$:
\begin{eqnarray}
\delta(B(f(x_0),r))&\leq&\mu_{f^{-1}B(f(x_0),r)}\left(\tau_{B(f(x_0),r)}^f\leq g\right)+2\frac{\pf({B(f(x_0),r)}\backslash {B(f(x_0),r-\rho)})}{\pf({B(f(x_0),r)})}\nonumber\\
&+&\frac{K}{\rho}\frac{\theta_g}{\pf({B(f(x_0),r)})}+\mu\left(\tau_{B(f(x_0),r)}^f\leq g\right).\label{eqbornedelta}
\end{eqnarray}
To prove that $\delta(B(f(x_0),r))\underset{r\rightarrow0}{\longrightarrow}0$ for $\mu$-almost every $x_0\in X$ such that $\underline{d}_{f_*\mu}(f(x_0))>0$ we need the two following lemmas:

\begin{lemma}\label{lemhit}
For every $x_0\in X$ such that $\underline{d}_{f_*\mu}(f(x_0))>0$, for any $d\in(0,\underline{d}_{f_*\mu}(f(x_0)))$, we have
\[\mu\left(\tau^f_r(x,x_0)\leq r^{-d}\right)\rightarrow0\qquad as\,\,r\rightarrow0.\]
\end{lemma}
\begin{proof}
One can observe that for any measurable subset $A\subset X$ and for any $n\in N^*$, we have
\begin{eqnarray*}
\mu(\tau_A\leq n)&=&\mu\left(T^{-1}A\cup T^{-2}A\cup\dots\cup T^{-n}A\right)\\
&\leq& \mu\left(T^{-1}A\right)+ \mu\left(T^{-2}A\right)+\dots+ \mu\left( T^{-n}A\right)\\
&\leq&n\mu(A).
\end{eqnarray*}
This implies that for every $x_0\in X$ such that $\underline{d}_{f_*\mu}(f(x_0))>0$ and for any $d\in(0,\underline{d}_{f_*\mu}(f(x_0)))$, we have
\begin{eqnarray*}
\mu\left(\tau^f_r(x,x_0)\leq r^{-d}\right)&=&\mu\left(\tau_{f^{-1}B(f(x_0),r)}\leq r^{-d}\right)\\
&\leq&r^{-d} \mu\left(f^{-1}B(f(x_0),r) \right)=r^{-d} f_*\mu\left(B(f(x_0),r) \right).
\end{eqnarray*}
Since $0<d<\underline{d}_{f_*\mu}(f(x_0))$, $r^{-d} f_*\mu\left(B(f(x_0),r)\right)\rightarrow0$ as $r\rightarrow0$ and the lemma is proved.
\end{proof}

\begin{lemma}\label{lemshort} Under the assumptions of Theorem~\ref{thprinc}, for $\mu$-almost every $x_0\in X$ such that $\underline{d}_{f_*\mu}(f(x_0))>0$, for any $d\in(0,\underline{d}_{f_*\mu}(f(x_0)))$, we have
\begin{equation}\label{eqlemshort}
\mu_{f^{-1}B(f(x_0),r)}\left(\tau^f_r(x,x_0)\leq r^{-d}\right)\rightarrow0\qquad as\,\,r\rightarrow0.
\end{equation}
\end{lemma}

\begin{proof}
For $a>0$, let us define $Y_a=\{y\in \R^N,\underline{d}_{f_*\mu}(y)>a\}$. 
One can observe that Theorem~\ref{theoprfo} gives us that
\[
\liminf_{r\to0}\frac{\log\tau_r^f(x)}{-\log r}\ge \underline{d}_{f_*\mu}(f(x))>a 
\]
for $\mu$-a.e. $x\in f^{-1}(Y_a)$.
Let $r_0>0$ and define for $y\in Y_a$
\[A(r_0,y)=\{x\in X\colon f(x)=y\textrm{ and } \exists r<r_0, \tau_{2r}^f(x)<r^{-a/2}\}.\]
Let $\varepsilon>0$ and set
\[
D_\varepsilon(r_0)=\{y\in Y_a\colon \mu(A(r_0,y))\le \varepsilon\}.
\]
Let $x_0\in X$ such that $f(x_0)$ is a Lebesgue density point of the set $D_\varepsilon(r_0)$ for the measure $f_*\mu$, i.e.
\[
\frac{f_*\mu(B(f(x_0),r)\cap D_\varepsilon(r_0))}{f_*\mu(B(f(x_0),r))}\to 1
\]
as $r\to0$. Hence there exists $r_1<r_0$ such that for any $r<r_1$
\[
f_*\mu(B(f(x_0),r)\cap D_\varepsilon(r_0)^c) \le \varepsilon f_*\mu(B(f(x_0),r)).
\]
Let $r<r_1$ and $d>a$. We get

\begin{eqnarray*}
& &\mu_{f^{-1}B(f(x_0),r)}\left(\tau^f_r(x,x_0)\leq r^{-d}\right)\\
&=&\frac{1}{f_*\mu(B(f(x_0),r))}\int_{X} 1_{f^{-1}B(f(x_0),r)}(x) 1_{\{\tau^f_{B(f(x_0),r)}\le r^{-d}\}}(x) d\mu(x)\\
&\le&\frac{1}{f_*\mu(B(f(x_0),r))}\int_{X} 1_{B(f(x_0),r)}( f(x)) 1_{\{\tau_{2r}^f<r^{-a/2}\}}(x) d\mu(x)\\
&=&\frac{1}{f_*\mu(B(f(x_0),r))}\int_{X} 1_{B(f(x_0),r)}( f(x)) \E_\mu\left(1_{\{\tau_{2r}^f<r^{-a/2}\}}\big\vert f\right) d\mu(x)\\
&=&\frac{1}{f_*\mu(B(f(x_0),r))}\int_{\R^N} 1_{B(f(x_0),r)}( y) \E_\mu\left(1_{\{\tau_{2r}^f<r^{-a/2}\}}\big\vert f=y\right) df_*\mu(y)\\
&=&\frac{1}{f_*\mu(B(f(x_0),r))}
\int_{\R^N}1_{B(f(x_0),r)}(y) \mu(A(r_0,y))df_*\mu(y)\\
&\le&\frac{1}{f_*\mu(B(f(x_0),r))}\left(
f_*\mu(B(f(x_0),r)\cap D_\varepsilon(r_0)^c)+\varepsilon f_*\mu(B(f(x_0),r)\cap D_\varepsilon(r_0))\right)\\
&\le &\frac{1}{f_*\mu(B(f(x_0),r))}2\varepsilon f_*\mu(B(f(x_0),r))=2\varepsilon.
\end{eqnarray*}

Since $\varepsilon$ is arbitrary and the measure of $D_\varepsilon(r_0)$ can be made arbitrarily close to the measure of $Y_a$, 
this shows that $\mu_{f^{-1}B(f(x_0),r)}\left(\tau^f_r(x,x_0)\leq r^{-d}\right)\to0$ for $\mu$-a.e. $x_0\in f^{-1}(Y_a)$ and for any $d>a$. The lemma is prove since $a$ can be chosen arbitrary small.
\end{proof}

We now have all the ingredients to finish the proof of the lemma. Let $x_0\in X$ such that $\underline{d}_{f_*\mu}(f(x_0))>0$, such that $\overline{d}_{f_*\mu}(f(x_0))\leq N$ and such that \eqref{eqcour} and \eqref{eqlemshort} are satisfied.
Since the upper local dimension of a measure is almost-everywhere smaller than the dimension of the ambient space, we obtain that $\mu\left(x\in X:\overline{d}_{f_*\mu}(f(x))\leq N\right)=1$.

Let $0<d<\underline{d}_{f_*\mu}(f(x_0))$. Let us choose $g=\lfloor r^{-d}\rfloor$ and $\rho=\theta_{g/2}$.

The choice of $x_0$ and $g$ and Lemma~\ref{lemshort} give us that
\begin{equation}\label{eqestd1}
\mu_{f^{-1}B(f(x_0),r)}\left(\tau_{B(f(x_0),r)}^f\leq g\right)\rightarrow0\qquad as\,\,r\rightarrow0
\end{equation}
and by Lemma~\ref{lemhit}, we have
\begin{equation}\label{eqestd2}
\mu\left(\tau^f_r(x,x_0)\leq g\right)\rightarrow0\qquad as\,\,r\rightarrow0.
\end{equation}
Using \eqref{eqcour}, we have that for $r$ sufficiently small
\[f_*\mu\left(B\left(f(x_0),r\right)\backslash B\left(f(x_0),r-\rho\right)\right)\leq r^{-b}\rho^a\]
and 
\begin{equation}\label{eqai}
\pf\left(B\left(f(x_0),r\right)\right)\geq r^{N+1},
\end{equation}
since $x_0$ satisfied $\overline{d}_{f_*\mu}(f(x_0))\leq N$, which implies that
\begin{equation}\label{eqestd3}
\frac{\pf({B(f(x_0),r)}\backslash {B(f(x_0),r-\rho)})}{\pf({B(f(x_0),r)})}\rightarrow0\qquad as\,\,r\rightarrow0.
\end{equation}
The choice of $g$ and $\rho$ together with~\eqref{eqai} gives us
\begin{equation}\label{eqestd4}
\frac{\theta_g}{\rho.\pf({B(f(x_0),r)})}\rightarrow0\qquad as\,\,r\rightarrow0.
\end{equation}
Finally, using hypothesis~\eqref{hypcouronne} and \eqref{eqbornedelta} together with ~\eqref{eqestd1}, \eqref{eqestd2},\eqref{eqestd3} and \eqref{eqestd4}, we obtain that $\delta(B(f(x_0),r))\rightarrow0$ as $r\rightarrow0$ for $\mu$-almost every $x_0\in X$ such that $\underline{d}_{f_*\mu}(f(x_0))>0$ which concludes the proof of the lemma.
\end{proof}
\begin{proof}[Proof of Theorem~\ref{thprinc}]
Let $t>0$. Let us denoted $n=\left\lfloor\frac{t}{\pf(B(f(x_0),r))}\right\rfloor$ and $A=B(f(x_0),r)$. We observe that
\begin{eqnarray*}
& &\left|\mu\left(\tau_{B(f(x_0),r)}^f>\frac{t}{f_*\mu\left(B\left(f(x_0),r\right)\right)}\right)-e^{-t}\right|\\
&=&\left|\mu\left(\tau_A^f>n\right)-\left(1-\pf(A)\right)^n+\left(1-\pf(A)\right)^n-e^{-t}\right|\\
&\leq& \delta(A)+\left|\left(1-\pf(A)\right)^n-e^{-t}\right|
\end{eqnarray*}
Moreover, we have 
\begin{equation}\label{eqexp}
\left|\left(1-\pf(A)\right)^n-e^{-t}\right|\leq \left|\left(1-\pf(A)\right)^n-\left(1-\frac{t}{n}\right)^n\right|+\left|\left(1-\frac{t}{n}\right)^n-e^{-t}\right|.
\end{equation}
Using the mean value theorem and the definition of $n$, we obtain that
\begin{eqnarray}
\left|\left(1-\pf(A)\right)^n-\left(1-\frac{t}{n}\right)^n\right|&\leq& n\left|\pf(A)-\frac{t}{n}\right|\nonumber\\
&\leq& \frac{t}{n}.\label{eqtvm}
\end{eqnarray}
Since it is well-known that $\left|\left(1-\frac{t}{n}\right)^n-e^{-t}\right|\rightarrow0$ as $n\rightarrow\infty$, \eqref{eqexp} together with \eqref{eqtvm} implies that 
\begin{equation}\label{eqdif}
\left|\left(1-\pf(A)\right)^n-e^{-t}\right|\rightarrow0\qquad\textrm{ as } n\rightarrow\infty.
\end{equation}
Finally the first part of the theorem is proved using \eqref{eqdif} and since by Lemma~\ref{lemlimdelta}, $\delta(A)\rightarrow0$ as $r\rightarrow0$ for $\mu$-almost every $x_0\in X$ such that $\underline{d}_{f_*\mu}(f(x_0))>0$.

To prove the second part of the theorem, we just need to observe that
\begin{eqnarray*}
\left|\mu_{f^{-1}A}\left(\tau_A^f>n\right)-e^{-t}\right|&\leq& \left|\mu_{f^{-1}A}\left(\tau_A^f>n\right)-\mu\left(\tau_A^f>n\right)\right|+\left|\mu\left(\tau_A^f>n\right)-e^{-t}\right|\\
&\leq&\delta (A)+\left|\mu\left(\tau_A^f>n\right)-e^{-t}\right|
\end{eqnarray*}
and Lemma~\ref{lemlimdelta} and the first part of the theorem give us that the right hand side of the inequality goes to zero as $r$ goes to zero.
\end{proof}

%%%%%%%%%%%%
%%%%%%%%%%%%%%%%
\section{Proof of the exponential law for random dynamical systems}\label{proofrds}
%%%%%%%%%%%%%%%%
%%%%%%%%%%%%%%%%%%

In this section, we will prove Theorem~\ref{theorandomhit} and Theorem~\ref{theorandomhitiid}.
\begin{proof}[Proof of Theorem~\ref{theorandomhit}]
This theorem is proved using Theorem~\ref{thprinc} applied to the dynamical system $(\Omega\times X,\mathcal{B}(\Omega\times X),\mu ,S)$ with the observation $f$ defined by
\begin{eqnarray*}
f\,\,:\Omega\times X&\longrightarrow& X\\
(\omega,x)&\longmapsto& x.
\end{eqnarray*}
With this observation, for all $(\omega,x)\in\Omega\times X$ and for all $r>0$, we can link the hitting time for the observation and the hitting time for the random dynamical system
\[\tau^f_{B(f(x_0,\omega_0),r)}(\omega,x)=\tau^\omega_{B(x_0,r)}(x),\]
we can identify the pushforward measure
\[f_*\mu=\nu,\]
and for the pointwise dimensions we can observe that
\[\underline{d}_{f_*\mu}(f(x_0,\omega_0))=\underline{d}_\nu(x_0)\qquad\textrm{and}\qquad \overline{d}^f_{f_*\mu}(f(x_0,\omega_0)=\overline{d}_\nu(x_0).\]
Moreover, the random dynamical system is random aperiodic if and only if the system is $\mu$-almost aperiodic for $f$.

Finally, in the proof of Theorem~\ref{thprinc}, one can observe that hypothesis \eqref{decayskewrds} is sufficient to prove \eqref{eqdecay} and thus the theorem is proved.
\end{proof}

\begin{proof}[Proof of Theorem~\ref{theorandomhitiid}]
One can see that the difference between Theorem~\ref{theorandomhitiid} and Theorem~\ref{theorandomhit} lies in the decay of correlations and that the decay of correlations is only used to obtain equation \eqref{eqdecay}. Thus, we will prove that equation \eqref{eqdecay} is still satisfied under the condition of Theorem~\ref{theorandomhitiid}. 

As observe in the proof of Theorem~\ref{theorandomhit}, the observable $\phi(\omega,x)$ does not depend on $\omega$ and can be substituted by an observable $\varphi(x)$. Then, using the setting of Theorem~\ref{theorandomhitiid}, we have
\begin{eqnarray*}
\int\phi.1_{B_{n-g}}\circ T^g d\mu&=&\int_X\varphi(x)\int_{\Lambda^\N}1_{\left\{\tau_{r}^{\sigma^g\underline{\lambda}}(T^g_{\underline{\lambda}}x,x_0)>n-g\right\}}d\mathcal{P}^\N d\nu\\
&=& \int_X\varphi(x)\int_{\Lambda^\N}\int_{\Lambda^\N}1_{\left\{\tau_{r}^{\widetilde{\underline{\lambda}}}(T^g_{\underline{\lambda}}x,x_0)>n-g\right\}}d\mathcal{P}^\N(\widetilde{\underline{\lambda}})d\mathcal{P}^\N(\underline{\lambda}) d\nu.
\end{eqnarray*}
Indeed, the fact that the $(\lambda_i)$ are chosen i.i.d. gives us
\begin{eqnarray*}
\int_{\Lambda^\N}1_{\left\{\tau_{r}^{\sigma^g\underline{\lambda}}(T^g_{\underline{\lambda}}x,x_0)>n-g\right\}}d\mathcal{P}^\N&=&\int_{\Lambda^n}1_{\left\{\tau_{r}^{\sigma^g\underline{\lambda}}(T^g_{\underline{\lambda}}x,x_0)>n-g\right\}}d\mathcal{P}^n\\
&=&\int_{\Lambda^n}1_{\left\{T^{g+1}_{\underline{\lambda}}x\notin B,\dots,T^n_{\underline{\lambda}}x\notin B \right\}}d\mathcal{P}^n\\
&=&\int_{\Lambda^n}1_{\left\{T_{\lambda_{g+1}}\circ T_{\lambda_g}\circ \dots \circ T_{\lambda_1}x\notin B,\dots,T_{\lambda_{n}}\circ\dots \circ T_{\lambda_1}x\notin B \right\}}d\mathcal{P}^n\\
&=&\int_{\Lambda^g}\int_{\Lambda^{n-g}}1_{\left\{T_{\tilde\lambda_1}\circ T^g_{\underline{\lambda}}x\notin B,\dots,T_{\tilde\lambda_{n-g}}\circ\dots\circ T_{\tilde\lambda_1}\circ T^g_{\underline{\lambda}}x\notin B \right\}}d\mathcal{P}^{n-g}(\widetilde{\underline{\lambda}})d\mathcal{P}^g(\underline{\lambda})\\
&=&\int_{\Lambda^g}\int_{\Lambda^{n-g}}1_{\left\{\tau_{r}^{\widetilde{\underline{\lambda}}}(T^g_{\underline{\lambda}}x,x_0)>n-g\right\}}d\mathcal{P}^{n-g}(\widetilde{\underline{\lambda}})d\mathcal{P}^g(\underline{\lambda})\\
&=&\int_{\Lambda^\N}\int_{\Lambda^\N}1_{\left\{\tau_{r}^{\widetilde{\underline{\lambda}}}(T^g_{\underline{\lambda}}x,x_0)>n-g\right\}}d\mathcal{P}^\N(\widetilde{\underline{\lambda}})d\mathcal{P}^\N(\underline{\lambda})
\end{eqnarray*}
where $B$ stands for $B(x_0,r)$.
Thus, we obtain that
\[\int\phi.1_{B_{n-g}}\circ T^g d\mu
= \int_X\varphi(x)\int_{\Lambda^\N}\psi(T^g_{\underline{\lambda}}x)d\mathcal{P}^\N(\underline{\lambda}) d\nu\]
where 
\[\psi(x)=\int_{\Lambda^\N}1_{\left\{\tau_{r}^{\widetilde{\underline{\lambda}}}(x,x_0)>n-g\right\}}d\mathcal{P}^\N(\widetilde{\underline{\lambda}}).\]
Since one can easily observe that
\[\mu(B_{n-g})=\int_X\psi(x)d\nu,\]
we can use the hypothesis on the decay of correlations for the random dynamical systems to obtain the equivalent of equation \eqref{eqdecay} in this setting
\begin{eqnarray*}
\left|\int\phi.1_{B_{n-g}}\circ T^g d\mu-\mu(B_{n-g})\int\phi d\mu\right|&=&\int_X\varphi(x)\int_{\Lambda^\N}\psi(T^g_{\underline{\lambda}}x)d\mathcal{P}^\N d\nu-\int_X\varphi d\nu\int_X\psi d\nu\\
&\leq&\|\varphi\|_{Lip}\|\psi\|_{\infty}.\theta_g \\
&\leq& \frac{K}{\rho}\theta_g
\end{eqnarray*}
and the theorem is proved as a corollary of Theorem~\ref{theorandomhit}.
\end{proof}

\section*{Acknowledgements} The author would like to thank B. Saussol for his useful help and comments. The author is grateful to the anonymous referees for their helpful comments.

\bibliographystyle{siam} 
\bibliography{biblio-expo}

\end{document}